\documentclass[11pt,a4paper]{article}

\usepackage[utf8]{inputenc} 
\usepackage[T1]{fontenc}
\usepackage[english]{babel}
\usepackage{amsmath,amssymb,amsthm,mathtools}
\usepackage{enumitem}
\usepackage{microtype}
\usepackage[top=2cm, bottom=1.5cm, left=2cm, right=1cm]{geometry}
\usepackage{xcolor}
\usepackage{hyperref}

\hypersetup{
  colorlinks=true,
  linkcolor=blue,
  citecolor=blue,
  urlcolor=blue
}

\newcommand{\Fin}{\mathrm{Fin}}
\newcommand{\FProd}{\mathrm{FProd}}
\newcommand{\MS}{\mathbb{N}^{(I)}}
\newcommand{\ProdEnum}[2]{\prod_{k=1}^{#1} #2}

\theoremstyle{definition}
\newtheorem{defn}{Definition}[section]
\newtheorem{ex}{Example}[section]

\theoremstyle{plain}
\newtheorem{thm}{Theorem}[section]
\newtheorem{lem}{Lemma}[section]
\newtheorem{prop}{Proposition}[section]
\newtheorem{cor}{Corollary}[section]

\theoremstyle{remark}
\newtheorem{rem}{Remark}[section]

\title{Finite products in commutative monoids: well-definition, recursion on finite subsets, and why the empty product is $1$.}
\author{
João Victor Monteiros de Andrade \\
Department of Computer Science \\
University of Brasília \\
\texttt{andrade.monteiros@aluno.unb.br} \\
\texttt{jotandrade98@gmail.com}
\and
Leonardo Santos da Cruz \\
Department of Computer Science \\
University of Brasília \\
\texttt{santos-cruz.sc@aluno.unb.br}  
}
\date{}

\begin{document}
\maketitle


\begin{abstract}
The convention "empty product $=1$"  is ubiquitous in mathematics, but often appears without an explicit structural justification. This note provides a self-contained reference to this fact in the context of commutative monoids. We construct the product of an indexed family by a finite set, prove its enumeration independence, and show that it is uniquely characterized by a recursion scheme in $\Fin(I)$: value in the empty set and insertion rule of a new index. In particular, the value of the empty product is necessarily the neutral element $1$. We further record two complementary and independent justifications of this fact: one via the list-free monoid and another via distributive identities in semi-rings. Next, we formulate the same phenomenon in universal terms by means of the commutative multiset-free monoid of finite support. We also discuss partially commutative extensions, via trace monoids and heaps, and include brief applications in linear algebra, survival statistics, category theory, and analysis. The corresponding additive version recovers, by the same principle, the identity "empty sum $=0$".
\end{abstract}


\tableofcontents

\section{Introduction}

The expression $\prod_{x\in P} a(x)$ is unambiguous when $P$ is a finite list (with a given order).

For a finite \emph{set} $P$ without a prescribed order, however, the notation requires justification:
different enumerations can generate products in different orders. In commutative monoids,
associativity and commutativity guarantee the independence of the enumeration, making it
meaningful to directly define the product over finite sets.

The main focus of this text is to fill a recurring expository gap in the literature.

In many areas, the convention ``\emph{empty product $=1$}'' is adopted as a standard without any
structural justification, even though this value is in fact \emph{forced} by minimum algebraic requirements: recursive uniqueness, preservation of the identity element by homomorphisms, and consistency with distributive identities. In combinatorics and algebraic manipulations, the convention avoids awkward case distinctions and preserves identities in the limiting cases $n=0$ or $n=1$ \cite{GKP94,Stanley11}. In probability, statistics, and stochastic processes, it often appears as an explicit convention.\cite{MicloZhang21,Gaunt18,Mance15,ErnstReinertSwan22}. 
In theoretical computer science, this appears side by side with the ``empty sum $=0$'' in inductive arguments \cite{EsparzaKuceraMayr06}. In category theory, the analogous statement is that the product of an empty family is a terminal object (whenever such a product is defined) \cite{Riehl16}. Our goal is to show that all these occurrences reflect the same simple structural principle, formalized via finite products in commutative monoids. To this end, we work with $\Fin(I)$, which denotes (for convenience) the set of all finite subsets of an index set $I$, that is, $\Fin(I)=\{P\subseteq I:\ |P|<\infty\}.$ More precisely, we show that the assignment $P \mapsto \prod_{x\in P} a(x)$ over finite sets is uniquely characterized by a recursion scheme on $\Fin(I)$: (i) in the empty case the value is $1$, and (ii) when inserting a new index $x\notin P$, the product is updated by multiplying by $a(x)$. In particular, this characterization necessarily implies $\prod_{\varnothing} a(x)=1.$ The construction and well-definedness are presented in Section~\ref{sec:well-defined}, and the ``empty + insertion'' clauses together with the uniqueness theorem are developed in Sections~\ref{sec:empty-insert}--\ref{sec:unique}.

\medskip
\noindent
\textbf{The additive parallel.}
Although the main focus of this note is multiplicative, the same formal mechanism has an entirely parallel additive counterpart. If $(A,+,0)$ is a commutative additive monoid, then the assignment $P\mapsto \sum_{x\in P} a(x)$ on finite subsets is equally characterized by a recursion scheme in $\Fin(I)$: the empty case has value $0$, and the insertion of a new index updates the sum by adding the new term. In particular, the identity ``\emph{empty sum $=0$}'' should not be seen as a mere independent notational convention, but as the additive version of the same structural principle that forces ``\emph{empty product $=1$}''. This symmetry will appear explicitly later, both in the Corollary ~\ref{cor:empty-sum} and in concrete examples, and helps to situate the phenomenon discussed here within a more general algebraic framework.

\medskip
\noindent
\textbf{The connection with algebra and the "why" of $1$.}
In algebra, the same idea reappears whenever substructures are required to be closed under finite products: if “closed under finite products” includes the case of zero factors, then the identity element must automatically belong to the substructure. This observation was explored by Poonen, who interprets the presence of (1) as part of the expected behavior of “finite products” in associative structures, including the product of an empty sequence \cite{Poonen19}; see also the complementary discussion by Conrad \cite{ConradRingdefs}. In categorical language, this appears in a structural way: a terminal object is precisely a null product (the product of an empty family), which explains the notation (1) for the case of arity zero \cite{AwodeyCT}. 

A substantial part of the central results of this article (well-definedness, the recursion principle, and the uniqueness theorem of Sections~\ref{sec:well-defined}--\ref{sec:unique}) was formally verified in the proof assistant Isabelle/HOL \cite{NipkowPaulsonWenzel2002}. The complete theory files accompany this text as supplementary material, and a reorganized version will be submitted to the \emph{Archive of Formal Proofs} (AFP) \cite{AFP}.

\medskip
\noindent
\textbf{Structure of the paper.}
\begin{enumerate}[label=(\arabic*)]
\item \S\ref{sec:well-defined}: construction of $\FProd$ and well-definition.
\item \S\ref{sec:empty-insert}--\S\ref{sec:unique}: induction in $\Fin(I)$ and uniqueness (empty + insertion).
\item \S\ref{sec:other-proofs}: two independent justifications for the empty product.
\item \S\ref{sec:multisets}: universal formulation via multisets (commutative free monoid).
\item \S\ref{sec:noncomm}--\S\ref{sec:traces}: non-commutative variations and partially commutative products (traces/heaps).
\item \S\ref{sec:apps}: four short applications (determinant $0\times0$, Kaplan-Meier,
empty product as a terminal object and infinite products in analysis) and a gallery of
explicit occurrences in the literature.
\end{enumerate}
\section{Preliminaries}

\begin{defn}[Commutative monoid]
A commutative multiplicative monoid is a triple $(A,\cdot,1)$ where $A$ is a set,
$\cdot:A\times A\to A$ is associative and commutative, and $1\in A$ is neutral:
$1\cdot a=a\cdot 1=a$ for all $a\in A$.

\end{defn}

\begin{defn}[Finite Subsets]
Let $I$ be a set. We denote by
\[
\Fin(I):=\{P\subseteq I : P \text{ is finite}\}
\]
the set of all finite subsets of $I$.
\end{defn}

\begin{defn}[Indexed Family]

Given a function $a:I\to A$, we write $a_i:=a(i)$ and call $(a_i)_{i\in I}$ a family in $A$.
\end{defn}

\section{Finite Product over Sets: Construction and Well-Definedness}\label{sec:well-defined}

\begin{defn}[Product via enumeration]\label{def:enumprod}
Let $P\in\Fin(I)$ with $|P| = n$, and let $e:\{1,\dots,n\}\to P$ be a bijection.

We define
\[
\Pi(a;e)\;:=\;\ProdEnum{n}{a_{e(k)}}\in A,
\]
with the convention $\Pi(a;e)=1$ when $n=0$ (that is, $P=\varnothing$).

\end{defn}

\begin{lem}[Invariance under permutation]\label{lem:perm}
Let $x_1,\dots,x_n\in A$. For every permutation $\sigma\in S_n$,

\[
x_1\cdots x_n \;=\; x_{\sigma(1)}\cdots x_{\sigma(n)}.
\]

\end{lem}

\begin{proof}
It is sufficient to treat adjacent transpositions, since every permutation is a product of them.

If $\tau$ swaps $i$ and $i+1$ and fixes the others, then by associativity
\[
x_1\cdots x_n
= (x_1\cdots x_{i-1})\cdot (x_i\cdot x_{i+1})\cdot (x_{i+2}\cdots x_n),
\]
and by commutativity $x_i\cdot x_{i+1}=x_{i+1}\cdot x_i$.

\end{proof}

\begin{prop}[Enumeration Independence]\label{prop:enum-indep}
Let $P\in\Fin(I)$ and let $e,e':\{1,\dots,n\}\to P$ be bijections.

Then $\Pi(a;e)=\Pi(a;e')$.

\end{prop}

\begin{proof}
Since $e,e'$ are bijections with the same codomain, there exists $\sigma\in S_n$ such that $e'=e\circ\sigma$.

Therefore,

\[
\Pi(a;e')=\ProdEnum{n}{a_{e' (k)}}=\ProdEnum{n}{a_{e(\sigma(k))}}
=\ProdEnum{n}{a_{e(k)}}=\Pi(a;e),
\]
where we used the Lemma~\ref{lem:perm}.

\end{proof}

\begin{defn}[Finite product over a set]\label{def:fprod}
For $P\in\Fin(I)$, we define
\[
\FProd(a,P)\;:=\;\Pi(a;e),
\]
where $e$ is any bijection $\{1,\dots,|P|\}\to P$.

The Proposition ~\ref{prop:enum-indep} guarantees that $\FProd(a,P)$ is well-defined.

\end{defn}

\begin{ex}[Example in $(\mathbb{R},\cdot,1)$]
Consider the commutative monoid $(A,\cdot,1)=(\mathbb{R},\cdot,1)$, take $I=\mathbb{N}$ and define
$a:I\to A$ by $a(i)=i+1$.
For $P=\{1,3,4\}\in\Fin(I)$, we obtain
\[
\FProd(a,P)=a(1)\,a(3)\,a(4)=2\cdot 4\cdot 5=40,
\]
regardless of the enumeration chosen for $P$.
\end{ex}


\section{Empty case and insertion step}\label{sec:empty-insert}

\begin{lem}[Empty product]\label{lem:empty}
\[
\FProd(a,\varnothing)=1.
\]
\end{lem}

\begin{proof}
If $P=\varnothing$, then $|P|=0$ and the convention in the Definition~\ref{def:enumprod} gives $\FProd(a,\varnothing)=1$.
\end{proof}

\begin{lem}[Insertion step]\label{lem:insert}
If $P\in\Fin(I)$ e $x\in I\setminus P$, then
\[
\FProd(a,P\cup\{x\})=\FProd(a,P)\cdot a(x).
\]
\end{lem}

\begin{proof}
Let $n=|P|$ and take a bijection $e:\{1,\dots,n\}\to P$. Define $e^+:\{1,\dots,n+1\}\to P\cup\{x\}$
by $e^+(k)=e(k)$ for $1\le k\le n$ and $e^+(n+1)=x$. Then
\[
\FProd(a,P\cup\{x\})
=\Pi(a;e^+)
=\left(\ProdEnum{n}{a_{e(k)}}\right)\cdot a(x)
=\FProd(a,P)\cdot a(x).
\]
\end{proof}

\section{Induction in \texorpdfstring{$\Fin(I)$}{Fin(I)} and uniqueness}\label{sec:unique}

\begin{lem}[Induction in  $\Fin(I)$]\label{lem:fin-ind}
Let $\mathcal{S}\subseteq \Fin(I)$ be such that:
\begin{enumerate}[label=(\alph*)]
\item $\varnothing\in \mathcal{S}$;
\item if $Q\in \mathcal{S}$ e $x\in I\setminus Q$, then $Q\cup\{x\}\in \mathcal{S}$.
\end{enumerate}
Then $\mathcal{S}=\Fin(I)$.
\end{lem}

\begin{proof}
By induction on $n=|P|$. For $n=0$, $P=\varnothing\in\mathcal{S}$. For the step, given $|P|=n+1$ choose
$x\in P$ and set $Q=P\setminus\{x\}$. Then $|Q|=n$ and $Q\in\mathcal{S}$; since $x\notin Q$, it follows
$Q\cup\{x\}=P\in\mathcal{S}$.
\end{proof}

\begin{thm}[Recursion/Uniqueness of the finite product]\label{thm:unique}
Let $(A,\cdot,1)$ be a commutative monoid, $I$ a set and $a:I\to A$.
If $f:\Fin(I)\to A$ satisfies:
\begin{enumerate}[label=(\arabic*)]
\item $f(\varnothing)=1$;
\item for all $Q\in\Fin(I)$ e $x\in I\setminus Q$,
\[
f(Q\cup\{x\})=f(Q)\cdot a(x),
\]
\end{enumerate}
then, for all $P\in\Fin(I)$,
\[
f(P)=\FProd(a,P).
\]
\end{thm}

\begin{proof}
Consider $\mathcal{S}:=\{P\in\Fin(I): f(P)=\FProd(a,P)\}$.
We have $\varnothing\in\mathcal{S}$ by (1) and Lemma~\ref{lem:empty}.
If $Q\in\mathcal{S}$ and $x\notin Q$, then by (2) and by Lemma~\ref{lem:insert},
\[
f(Q\cup\{x\})=f(Q)\cdot a(x)=\FProd(a,Q)\cdot a(x)=\FProd(a,Q\cup\{x\}),
\]
therefore $Q\cup\{x\}\in\mathcal{S}$. By Lemma~\ref{lem:fin-ind}, $\mathcal{S}=\Fin(I)$.
\end{proof}

\begin{cor}[Recognition Criterion]\label{cor:criterion}
To prove that a quantity $F(P)$ is necessarily a finite product of the factors $a(x)$,
it suffices to verify $F(\varnothing)=1$ and $F(P\cup\{x\})=F(P)\cdot a(x)$ for $x\notin P$.
\end{cor}

\begin{prop}[Characterization via disjoint union]\label{prop:disjoint-char}
Let $(A,\cdot,1)$ be a commutative monoid, $I$ a set, and $F:\Fin(I)\to A$ a function that satisfies:
\begin{enumerate}[label=(\alph*)]
\item $F(\varnothing)=1$;
\item if $P,Q\in\Fin(I)$ are disjoint, then $F(P\cup Q)=F(P)\cdot F(Q)$.
\end{enumerate}
Define $a:I\to A$ by $a(x):=F(\{x\})$. Then, for all
$P\in\Fin(I)$,
\[
F(P)=\FProd(a,P).
\]
In particular, $F$ is uniquely determined by the values in singletons.
\end{prop}

\begin{proof}
Let $P\in\Fin(I)$ and $x\in I\setminus P$. Then $P$ and $\{x\}$ are disjoint, therefore, by hypothesis \textup{(b)},
\[
F(P\cup\{x\})=F(P)\cdot F(\{x\})=F(P)\cdot a(x).
\]
Thus, $F$ satisfies the conditions of the Theorem~\ref{thm:unique} with this choice of $a$, and therefore
$F(P)=\FProd(a,P)$ for all $P\in\Fin(I)$.
\end{proof}

\begin{cor}[Empty Sum]\label{cor:empty-sum}
Let $(A, +, 0)$ be a commutative additive monoid. The unique assignment of sums over finite sets, $S: \Fin(I) \to A$, that satisfies $S(P \cup Q) = S(P) + S(Q)$ for disjoint sets and $S(\{x\}) = a(x)$ is characterized by:
\[
\sum_{x \in \varnothing} a(x) = 0.
\]
\end{cor}

\begin{proof}
The result follows immediately from Theorem~\ref{thm:unique} and Proposition~\ref{prop:disjoint-char} by duality of notation. It suffices to replace the multiplicative monoid $(A, \cdot, 1)$ with the additive monoid $(A, +, 0)$.
\end{proof}

\begin{ex}[Cardinality of disjoint sets]
Let $I$ be a set and $A = (\mathbb{N}, +, 0)$. If we define $a(x) = 1$ for all $x \in I$, the finite sum $S(P) = \sum_{x \in P} 1$ represents the cardinality $|P|$. For the property $|P \cup Q| = |P| + |Q|$ to hold when $P$ and $Q$ are disjoint, the case where $P = \varnothing$ forces:
\[
|\varnothing \cup Q| = |\varnothing| + |Q| \implies |Q| = |\varnothing| + |Q|,
\]
This requires $|\varnothing| = 0$. Thus, the empty sum of "ones" is zero.
\end{ex}

\begin{ex}[Linear combinations in vector spaces]
Let $V$ be a vector space over a field $K$. The set $V$ with the addition of vectors forms a commutative monoid $(V, +, 0)$. A linear combination of a family of vectors $\{v_x\}_{x \in P}$ is the finite sum $\sum_{x \in P} c_x v_x$. By definition and for consistency with subspaces (where the subspace spanned by the empty set is $\{\vec{0}\}$), the sum over an empty index set must result in the additive identity element:
\[
\sum_{x \in \varnothing} c_x v_x = \vec{0}.
\]
\end{ex}

\begin{ex}[Graph Theory and the Handshaking Lemma]
In a graph $G=(V,E)$, the degree of a vertex $v \in V$ is defined as the number of edges incident to it. If a vertex is isolated (has no edges), the set of incident edges is empty, and its degree is precisely the empty sum.
\[
\sum_{e \in \varnothing} 1 = 0.
\]
Without this convention, the uniform formulation of basic results in graph theory—such as the identity
\[
\sum_{v \in V} \deg(v) = 2|E|,
\]
known as the Handshake Lemma (\cite{bondy1976graph}) — would require artificial exceptions for isolated vertices.
\end{ex}

\begin{prop}[Homomorphism compatibility]\label{prop:hom-prod}
Let $(A,\cdot,1)$ and $(B,\star,1')$ be commutative monoids, $\varphi:A\to B$ a homomorphism of monoids
($\varphi(1)=1'$ and $\varphi(x\cdot y)=\varphi(x)\star\varphi(y)$) and $a:I\to A$ a family. For every $P\in\Fin(I)$, the following holds
\[
\varphi\bigl(\FProd(a,P)\bigr)\;=\;\FProd(\varphi\circ a,P),
\]
that is,
\[
\varphi\!\left(\prod_{x\in P} a(x)\right)
\;=\;
\prod_{x\in P} \varphi(a(x)).
\]
\end{prop}

\begin{proof}
Fix $P\in\Fin(I)$ and consider $f:\Fin(I)\to B$ given by $f(Q):=\varphi(\FProd(a,Q))$.
We have $f(\varnothing)=\varphi(1)=1'$.
If $Q\in\Fin(I)$ e $x\in I\setminus Q$, then
\[
f(Q\cup\{x\})
= \varphi\bigl(\FProd(a,Q\cup\{x\})\bigr)
= \varphi\bigl(\FProd(a,Q)\cdot a(x)\bigr)
= \varphi(\FProd(a,Q))\star\varphi(a(x))
= f(Q)\star (\varphi\circ a)(x).
\]
By Theorem ~\ref{thm:unique} (applied to the monoid $(B,\star,1')$ and the family $\varphi\circ a$),
it follows that $f(P)=\FProd(\varphi\circ a,P)$, that is, the desired identity.
\end{proof}

\section{Two independent justifications for the empty product}\label{sec:other-proofs}

The Theorem transforms "empty product $=1$" into part of the recursion principle. Next, we record two additional justifications (useful for readers from other fields), which also enforce the value $1$.

\subsection{List-free monoid (without commutativity)}

\begin{defn}[List-free monoid]
For a set $X$, let $X^\ast$ denote the set of finite lists (words) with letters in $X$,
with concatenation $\mathbin{+\!\!+}$ and empty word $\varepsilon$.
Then $(X^\ast,\mathbin{+\!\!+},\varepsilon)$ is a monoid.
\end{defn}

\begin{prop}[Strength list evaluation $\Pi(\varepsilon)=1$]\label{prop:list-proof}
Let $(A,\cdot,1)$ be a monoid (not necessarily commutative). There exists a unique monoid homomorphism
\[
\Pi:(A^\ast,\mathbin{+\!\!+},\varepsilon)\to (A,\cdot,1)
\]
such that $\Pi([a])=a$ for all $a\in A$. Necessarily, $\Pi(\varepsilon)=1$.
\end{prop}

\begin{proof}
Homomorphisms of monoids preserve the neutral element. Therefore $\Pi(\varepsilon)=1$.
\end{proof}

\subsection{Via distributivity in semirings}

\begin{defn}[Commutative semiring]
A \emph{Commutative semiring
} it is a quintuple $(R, +, 0, \cdot, 1)$ such that $(R, +, 0)$ and $(R, \cdot, 1)$ are commutative monoids, and $\cdot$ distributes over $+$.
\end{defn}

\begin{prop}[The constant term in $\prod(1+b)$ forces the empty product to be $1$]\label{prop:distrib-proof}
Let $(R,+,0,\cdot,1)$ be a commutative semiring. Suppose one wants an identity of the form
\[
\prod_{x\in P}\bigl(1+b(x)\bigr)=\sum_{S\subseteq P}\ \prod_{x\in S} b(x)
\]
to hold for every $P\in\Fin(I)$ and every family $b:P\to R$, where the product over $S=\varnothing$ is interpreted as a
constant element $c\in R$. Then necessarily $c=1$.
\end{prop}

\begin{proof}
Take $P=\{x\}$, a singleton. The left-hand side is $1+b(x)$.
The right-hand side sums the subsets. $\varnothing$ e $\{x\}$:
\[
1+b(x)=c+b(x),
\]
therefore $c=1$.
\end{proof}

\section{Multisets and the universal formulation (commutative free monoid)}\label{sec:multisets}

\begin{defn}[Finite support multisets]
Fix a set $I$. Define
\[
\MS:=\{m:I\to\mathbb{N}\;:\;\mathrm{supp}(m)\ \text{is finite}\},
\qquad \mathrm{supp}(m):=\{i\in I:m(i)\neq 0\}.
\]
With point-to-point summation and the null element $0$, $\MS$ is a commutative monoid. For $i\in I$, denote by $\delta_i$ the singleton multiset: $\delta_i(i)=1$ and $\delta_i(j)=0$ if $j\neq i$.
\end{defn}

\begin{thm}[Universal property: commutative free monoid]\label{thm:freecomm}
Let $(A,\cdot,1)$ be a commutative monoid and let $a:I\to A$ be a mapping.
Then there exists a unique homomorphism of commutative monoids.
\[
\Phi_a:(\MS,+,0)\to (A,\cdot,1)
\]
such that $\Phi_a(\delta_i)=a(i)$ for all $i\in I$.
Moreover, for $m\in\MS$,
\[
\Phi_a(m)=\prod_{i\in\mathrm{supp}(m)} a(i)^{\,m(i)},
\]
where $a(i)^k$ is the iterated product (and $a(i)^0:=1$).
\end{thm}

\begin{proof}
\emph{Existence.} Define $\Phi_a(0):=1$ and, for $m\neq 0$,
\[
\Phi_a(m):=\prod_{i\in\mathrm{supp}(m)} a(i)^{m(i)}.
\]
Since $\mathrm{supp}(m)$ is finite, the product makes sense because \S\ref{sec:well-defined}.
The verification of $\Phi_a(m+n)=\Phi_a(m)\cdot\Phi_a(n)$ is a finite regrouping using commutativity.

\emph{Uniqueness.} If $\Psi$ is another homomorphism with $\Psi(\delta_i)=a(i)$, then
$m=\sum_{i\in\mathrm{supp}(m)} m(i)\,\delta_i$ e
\[
\Psi(m)=\prod_{i\in\mathrm{supp}(m)} \Psi(\delta_i)^{m(i)}
=\prod_{i\in\mathrm{supp}(m)} a(i)^{m(i)}=\Phi_a(m).
\]
\end{proof}

\begin{cor}[Empty product as a way to preserve neutrality.]\label{cor:empty-forced}
In the situation of Theorem~\ref{thm:freecomm}, necessarily $\Phi_a(0)=1$.
Thus, ``empty product $=1$'' expresses preservation of the identity element of the free commutative monoid.
\end{cor}

\section{Variations in the non-commutative case}\label{sec:noncomm}

If $(A,\cdot,1)$ is not commutative, in general there is no intrinsic “product over a finite set”:
the order of multiplication becomes relevant. Still, two facts remain useful:
(i) the correct object is a \emph{list/word} (where the empty product is $1$ by Prop.~\ref{prop:list-proof});
(ii) for subsets whose elements commute pairwise, the product becomes independent of the enumeration again.

\begin{ex}[Failure for sets: matrices]\label{ex:noncomm-matrix}
Em $A=M_2(\mathbb{R})$, tome
\[
X=\begin{pmatrix}1&1\\0&1\end{pmatrix},\qquad
Y=\begin{pmatrix}1&0\\1&1\end{pmatrix}.
\]
So $XY\neq YX$, and therefore the "product" of $\{X,Y\}$ depends on the enumeration.
\end{ex}

\begin{prop}[Independence of order under pairwise commutation]\label{prop:pairwise-comm}
Let $x_1,\dots,x_n\in A$ be such that $x_ix_j=x_jx_i$ for all $i,j$. Then, for every permutation $\sigma\in S_n$,
\[
x_1\cdots x_n = x_{\sigma(1)}\cdots x_{\sigma(n)}.
\]
\end{prop}

\begin{proof}
As in the Lemma~\ref{lem:perm}, it suffices to handle adjacent transpositions using associativity and the commutation hypothesis.
\end{proof}

\section{Partially commutative products: trace monoids and heaps}\label{sec:traces}

This section expands on non-commutative variation in the partially commutative direction: some order changes are allowed, others are not. This leads to Mazurkiewicz trace monoids, which interpolate between lists and multisets.\cite{CartierFoata69,Diekert90,DiekertRozenberg95}.

\subsection{Commutation independence and congruence}

\begin{defn}[Independence relationship]
An \emph{independence alphabet} is a pair $(\Sigma,I)$, where $I\subseteq \Sigma\times\Sigma$ is symmetric and irreflexive.
We interpret $(a,b)\in I$ as “$a$ and $b$ can swap places”. Its complement is the dependence relation.
\end{defn}

\begin{defn}[Mazurkiewicz congruence and trace monoid]
Let $\equiv_I$ be the smallest congruence on $\Sigma^\ast$ generated by the relations

\[
u a b v \equiv_I u b a v \qquad (u,v\in \Sigma^\ast,\ (a,b)\in I).
\]
The \emph{trace monoid} it is the quotient
\[
\mathbb{M}(\Sigma,I):=\Sigma^\ast/{\equiv_I},
\]
with concatenation-induced operation and neutral to the empty word class.
\end{defn}

\begin{ex}[Minimal example]\label{ex:trace-min}
Let $\Sigma=\{a,b,c\}$ and $I=\{(a,b),(b,a)\}$. Then $abc \equiv_I bac$, but $acb \not\equiv_I cab$.
\end{ex}

\subsection{Product well defined by trait: universal property}

\begin{thm}[Universality of the trace monoid]\label{thm:trace-universal}
Let $(A,\cdot,1)$ be a monoid and $w:\Sigma\to A$ a map such that
\[
(a,b)\in I \ \Rightarrow\ w(a)\,w(b)=w(b)\,w(a).
\]
Then there exists a unique monoid homomorphism
\[
\mathrm{Ev}_w:\mathbb{M}(\Sigma,I)\to A
\]
such that $\mathrm{Ev}_w([x])=w(x)$ for every $x\in\Sigma$ (where $[x]$ is the class of the one-letter word $x$).
In particular, $\mathrm{Ev}_w(1_{\mathbb{M}})=1$; that is, the \emph{empty trace} is sent to $1$.
\end{thm}

\begin{proof}
Define the evaluation $\widetilde{\mathrm{Ev}}:\Sigma^\ast\to A$ by
$\widetilde{\mathrm{Ev}}(x_1\cdots x_n)=w(x_1)\cdots w(x_n)$ e $\widetilde{\mathrm{Ev}}(\varepsilon)=1$.
The hypothesis guarantees that $\widetilde{\mathrm{Ev}}$ is constant in the classes of $\equiv_I$ (each generating step swaps factors
that commute in $A$). Therefore, it factors by the quotient, inducing $\mathrm{Ev}_w$. Uniqueness follows because $\mathbb{M}(\Sigma,I)$ is generated by the classes $[x]$.
\end{proof}

\subsection{Heaps: a partial analogue of "well-defined + recursion"}

A presentation consistent with the article's theme is to view features such as \emph{labeled finite posets} (Cartier-Foata heaps), where different equivalent words correspond to different linear extensions of the same poset.

\begin{defn}[Product by linear extension under switching of incomparables]\label{def:poset-product}
Let $(A,\cdot,1)$ be a monoid. Let $(P,\le)$ be a finite poset and $\ell:P\to A$ a labeling. A linear extension is a bijection $e:\{1,\dots,|P|\}\to P$ that preserves the order, that is, $e(i)<e(j)$ in $P$ implies $i<j$. For such $e$, define
\[
\Pi(\ell;e):=\prod_{k=1}^{|P|}\ell(e(k)).
\]
\end{defn}

\begin{lem}[Linear extension invariance]\label{lem:linear-extension-invariance}
Suppose that $\ell(x)\ell(y)=\ell(y)\ell(x)$ whenever $x$ and $y$ are incomparable in $P$. Then $\Pi(\ell;e)$ is independent of the linear extension $e$.
\end{lem}

\begin{proof}
We proceed by induction on $n=|P|$. For $n \le 1$ the result is immediate. Suppose $n \ge 2$ and that the result holds for all posets with fewer than $n$ elements.
Let $e$ and $e'$ be two linear extensions of $P$. It suffices to show that $\Pi(\ell;e)=\Pi(\ell;e')$. Let $m=e(n)$ and $m'=e'(n)$ be the elements placed in the last position by $e$ and $e'$, respectively; both are maximal elements of $P$.

\textit{Case 1: $m=m'$.} 
Then $e$ and $e'$ restrict to linear extensions $\bar e$ and $\bar e'$ of $P\setminus\{m\}$, 
and by the induction hypothesis $\Pi(\ell;\bar e)=\Pi(\ell;\bar e')$. Hence
\[
  \Pi(\ell;e)=\Pi(\ell;\bar e)\cdot\ell(m)
  =\Pi(\ell;\bar e')\cdot\ell(m)
  =\Pi(\ell;e').
\]

\textit{Case 2: $m\neq m'$.} 
Consider the sequence determined by $e'$:
\[
  p_1,\dots,p_n \quad\text{where}\quad p_k=e'(k).
\]
Let $k$ be such that $p_k=m$. Since $m$ is maximal, every element $p_j$ with $j>k$ 
is incomparable with $m$. Move $m$ to the last position by successive adjacent swaps: 
for $j=k,\dots,n-1$ swap $p_j$ and $p_{j+1}$. Each swap is valid (it swaps $m$ with 
an incomparable element) and preserves the linear extension. After these swaps we obtain 
a linear extension $e''$ that ends in $m$, and at each swap the product changes only 
by commuting two factors that commute by hypothesis. Therefore
\[
  \Pi(\ell;e')=\Pi(\ell;e'').
\]
Since $e$ and $e''$ both end in $m$, Case 1 gives $\Pi(\ell;e)=\Pi(\ell;e'')$. 
Combining the equalities, $\Pi(\ell;e)=\Pi(\ell;e')$.
\end{proof}

\begin{defn}[Heap evaluation]\label{def:heap-eval}
Under the hypotheses of the Lemma \ref{lem:linear-extension-invariance}, we define
\[
\mathrm{HeapProd}(P,\le,\ell):=\Pi(\ell;e)
\]
for any linear extension $e$. (We slightly abuse the notation, understanding $\le$ and $\ell$ also as their restrictions to the subposets.)
\end{defn}

\begin{prop}[Recursion by maximal removal]\label{prop:heap-recursion}
In the above scenarios, let $m\in P$ be a maximal element. Then
\[
\mathrm{HeapProd}(P,\le,\ell)=\mathrm{HeapProd}(P\setminus\{m\},\le,\ell)\cdot \ell(m),
\]
with the convention $\mathrm{HeapProd}(\varnothing)=1$.
\end{prop}

\begin{proof}
Since $m$ is maximal, there exists a linear extension that terminates at $m$. In this extension,
the product is written as the product of the prefix times $\ell(m)$. By the well-definition
(Lemma~\ref{lem:linear-extension-invariance}), the value is independent of the chosen extension.
\end{proof}

\begin{rem}
The Proposition ~\ref{prop:heap-recursion} is the partial analogue of the "insertion step":
The "step" now respects a partial order (concurrency). In particular, the empty case $=1$ remains unavoidable.
\end{rem}

\section{Applications and occurrences}\label{sec:apps}

\subsection{Application 1: Diagonal determinant and the case \texorpdfstring{$0\times 0$}{0x0}}

\begin{prop}[Diagonal matrix determinant]\label{prop:det-diag}
Let $K$ be a commutative ring and $D=\mathrm{diag}(d_1,\dots,d_n)$ a diagonal matrix.
Then, for every $n\ge 0$,
\[
\det(D)=\prod_{i=1}^n d_i,
\]
and in particular $\det([])=1$ when $n=0$.
\end{prop}

\begin{proof}
Consider $I=\{1,\dots,n\}$, $A=K$, and $a(i)=d_i$.
Define $F:\Fin(I)\to K$ by $F(P)=\det(\mathrm{diag}(u_1,\dots,u_n))$, where
$u_i=d_i$ if $i\in P$ and $u_i=1$ otherwise.
Then $F(\varnothing)=\det(I_n)=1$, and if $x\notin P$, activating $d_x$ multiplies the determinant by $d_x$,
that is, $F(P\cup\{x\})=F(P)\cdot d_x$.
By Theorem~\ref{thm:unique}, $F(P)=\FProd(a,P)$; taking $P=I$ we obtain the formula.
If $n=0$, then $I=\varnothing$ and $F(\varnothing)=1$, that is, $\det([])=1$.
\end{proof}

\subsection{Application 2: Kaplan-Meier via uniqueness (product-limit)}

\begin{prop}[Kaplan-Meier as a product characterized by updating]\label{prop:KM}
Consider distinct event times $t_1<\cdots<t_m$.
For each $j$, let $d_j$ be the number of events at $t_j$ and $n_j$ the number at risk immediately before $t_j$.
Define $I:=\{1,\dots,m\}$ and $a:I\to[0,1]$ by
\[
a(j):=1-\frac{d_j}{n_j}.
\]
For $t\in\mathbb{R}$, set $P(t):=\{\,j\in I:\ t_j\le t\,\}\in\Fin(I)$ and define
\[
\widehat S(t):=\FProd(a,P(t))=\prod_{t_j\le t}\left(1-\frac{d_j}{n_j}\right).
\]
Then:
\begin{enumerate}[label=(\alph*)]
\item (\textbf{base}) $\widehat S(t)=1$ for $t<t_1$;
\item (\textbf{step}) for $j\in I$, we have
$\widehat S(t_j)=\widehat S(t_{j-1})\cdot\left(1-\frac{d_j}{n_j}\right)$,
with the convention $\widehat S(t_0):=1$.
\end{enumerate}
\end{prop}

\begin{proof}
If $t<t_1$, then $P(t)=\varnothing$ and $\widehat S(t)=\FProd(a,\varnothing)=1$ (Lemma~\ref{lem:empty}).
For $t=t_j$, we have $P(t_j)=P(t_{j-1})\cup\{j\}$ and the Lemma~\ref{lem:insert} gives the update.
\end{proof}

\begin{rem}
The literature frequently mentions explicitly the convention "empty product taken as 1" in technical arguments related to the estimator \cite{MallerZhou93}, in addition to the original article \cite{KaplanMeier58}.
\end{rem}

\subsection{Application 3: Empty product in category theory (terminal object)}

\begin{prop}[The product over an empty family is a terminal object]\label{prop:cat-terminal}
Let $\mathcal{C}$ be a category with finite products. Then the product of an empty family (when defined)
is a terminal object of $\mathcal{C}$.
\end{prop}

\begin{proof}
By definition, a product $\prod_{i\in J} X_i$ is characterized by the universal property:
for every object $Y$, giving a morphism $Y\to \prod_{i\in J} X_i$ is equivalent to giving a family of morphisms $(Y\to X_i)_{i\in J}$. If $J=\varnothing$, such a family is empty data, hence there exists exactly one morphism $Y\to \prod_{\varnothing} X_i$ for every $Y$. This is precisely the definition of a terminal object.
\end{proof}

\subsection{Application 4: Infinite products in analysis}\label{subsec:inf-prod}

\begin{prop}[Infinite products and finite case normalization]\label{prop:inf-prod}
Let $(a_n)_{n\ge 1}$ be a sequence in $\mathbb{C}$ and, for $N\ge 0$, define the partial product
\[
P_N := \prod_{n=1}^N (1+a_n),
\]
with the convention $P_0:=1$. Then:
\begin{enumerate}[label=(\alph*)]
\item For every $1\le m\le N$,
\[
P_N \;=\; P_{m-1}\cdot \prod_{n=m}^N (1+a_n),
\]
where the product on the right is interpreted as the product over the finite subset $\{m,\dots,N\}\subseteq\mathbb{N}$.
\item If there exists $N_0\ge 0$ such that $a_n=0$ for all $n>N_0$, then the infinite product
\[
\prod_{n=1}^\infty (1+a_n)
\]
converges and coincides with the finite product $\prod_{n=1}^{N_0} (1+a_n)$. In particular, in the case $N_0=0$,
one obtains $\prod_{n=1}^\infty (1+a_n)=1$, that is, the empty product has value $1$.
\end{enumerate}
\end{prop}

\begin{proof}
For each $N \ge 0$, write
\[
P_N \;=\; \FProd(a,\{1,\dots,N\}),
\]
where $a(n) := 1 + a_n$ and, by convention, $\{1,\dots,0\} := \varnothing$. Let $1 \le m \le N$ and set
$P := \{1,\dots,m-1\}$ and $Q := \{m,\dots,N\}$; then $P, Q \in \Fin(\mathbb{N})$ are disjoint and $P \cup Q = \{1,\dots,N\}$.
By Proposition~\ref{prop:disjoint-char},
\[
P_N
= \FProd(a, P \cup Q)
= \FProd(a, P)\cdot \FProd(a, Q)
= P_{m-1}\cdot \prod_{n=m}^N (1+a_n),
\]
which proves \textup{(a)}.

For \textup{(b)}, if $a_n = 0$ for $n > N_0$, then $(1 + a_n) = 1$ for $n > N_0$, hence
\[
P_N = \prod_{n=1}^N (1+a_n) = \prod_{n=1}^{N_0} (1+a_n)
\]
for all $N \ge N_0$. Thus, the sequence $(P_N)_{N \ge 0}$ is eventually constant, and the infinite product
(defined as the limit of $P_N$) converges to $\prod_{n=1}^{N_0} (1+a_n)$. In the extreme case $N_0 = 0$, this says
precisely that the infinite product equals $P_0 = 1$, that is, the empty product has value $1$.
\end{proof}

\begin{rem}
In real and complex analysis, an \emph{infinite product} is defined precisely as the limit of partial products
$(P_N)_{N\ge 0}$. The convention $P_0=1$ causes degenerate cases (with only a finite number of factors
different from $1$) to fit naturally into the theory, without ad hoc distinctions. The Proposition ~\ref{prop:inf-prod} shows that this convention is compatible with the description via $\FProd$ and with the decomposition into finite blocks.
\end{rem}

\subsection{Explicit occurrences of the convention in the literature}

To make this note useful as a \emph{citable reference}, we record peer-reviewed examples in which the convention
is stated explicitly:
\begin{itemize}
\item Markov processes / pure-birth processes: \cite{MicloZhang21};
\item combinatorial expansions in probability: \cite{ErnstReinertSwan22};
\item products of random variables: \cite{Gaunt18};
\item normality / Hausdorff dimension: \cite{Mance15};
\item probabilistic verification in computer science: \cite{EsparzaKuceraMayr06}.
\end{itemize}

\section{Conclusion}

Starting from commutative monoids, we construct the finite product directly on sets and prove that it is simultaneously well-defined and unique under the pair (empty case + insertion step). The Theorem crystallizes this fact as a true principle of recursion in $\Fin(I)$: every function $f:\Fin(I)\to A$ that respects these two clauses necessarily coincides with $\FProd(a,\cdot)$. In this way, the convention "empty product $=1$" ceases to be an isolated addendum and becomes part of the recursive scheme itself that characterizes finite products.

We also saw that the value $1$ is unavoidable through completely independent means: as preservation of the neutral element in the free monoid of lists (without requiring commutativity), as a mandatory constant term in distributive identities of semi-rings, and as preservation of the neutral element in the commutative free monoid generated by multisets of finite support. The extension to partially commutative products via traces and heaps shows that the same triad—well-definedness, recursion, and neutral element—continues to operate even when commutativity is relaxed in a controlled manner.

Finally, it is worth emphasizing that the additive version discussed in the text is not merely a notational parallel, but the exact structural counterpart of the entire construction developed here. By replacing the multiplicative monoid $(A,\cdot,1)$ with the additive monoid $(A,+,0)$, the same recursion scheme in $\Fin(I)$ shows that the finite sum over sets is equally well-defined and uniquely determined by its basic clauses, which necessarily enforces the identity $\sum_{x\in\varnothing} a(x)=0$. Thus, "empty product $=1$" and "empty sum $=0$" should be understood together: both express the preservation of the identity element when moving from a binary operation to a finite aggregation indexed by sets.

Finally, applications in linear algebra (determinant $0\times 0$), survival statistics (Kaplan-Meier estimator), category theory (empty product as a terminal object), and analysis (infinite products), along with the gallery of explicit occurrences in the literature, reinforce the central message: "empty product $=1$" is not a mere convention of convenience, but the natural and systematic manifestation of a deep algebraic principle. By formulating finite products through recursion in $\Fin(I)$, this principle becomes transparent and, above all, inevitable.

\section*{Acknowledgments}

\bibliographystyle{plain}
\bibliography{referencias}

\end{document}